\documentclass[12pt,leqno]{amsart}
\usepackage{latexsym,amsmath,amssymb,amscd,xcolor,mathrsfs,marvosym,mathtools}
\usepackage[T1]{fontenc}
\usepackage[utf8]{inputenc}
\usepackage{graphicx,hyperref}

\usepackage[shortlabels]{enumitem}
\usepackage{wasysym}
\usepackage{scalerel,stackengine}
\usepackage[normalem]{ulem}
\usepackage{cleveref}

plus 6pt minus 12pt
\newcommand\reallywidetilde[1]{\ThisStyle{%
     \setbox0=\hbox{$\SavedStyle#1$}%
     \stackengine{-.1\LMpt}{$\SavedStyle#1$}{%
         \stretchto{\scaleto{\SavedStyle\mkern.2mu\AC}{.5150\wd0}}{.6\ht0}%
     }{O}{c}{F}{T}{S}%
}}

\newtheorem{theorem}{Theorem}
\newtheorem*{theorem*}{Theorem}
\newtheorem{lemma}[theorem]{Lemma}
\newtheorem{corollary}[theorem]{Corollary}

\newtheorem{question}[theorem]{Question}

\theoremstyle{definition}
\newtheorem{remark}[theorem]{Remark}
\newtheorem{definition}[theorem]{Definition}
\newtheorem{example}[theorem]{Example}

\newcommand{\overbarint}{
\rule[.036in]{.12in}{.009in}\kern-.16in \displaystyle\int }

\newcommand{\overbarcal}{\mbox{$ \rule[.036in]{.11in}{.007in}\kern-.128in\int $}}
\makeatletter
\def\overtilde#1{\mathop{\vbox{\m@th\ialign{##\crcr\noalign{\kern3\p@}%
      \sortoftildefill\crcr\noalign{\kern3\p@\nointerlineskip}%
      $\hfil\displaystyle{#1}\hfil$\crcr}}}\limits}

\def\sortoftildefill{$\m@th \setbox\z@\hbox{$\braceld$}%
  \braceld\leaders\vrule \@height\ht
  \z@ \@depth\z@\hfill\braceru$}
\makeatother

\newcommand{\bbbn}{\mathbb N}

\newcommand{\bbbr}{\mathbb R}
\newcommand{\eps}{\varepsilon}
\newcommand{\bbbs}{\mathbb S}
\newcommand{\bbbb}{\mathbb B}

\def\5{\text{\Saturn}}

\def\dist{\operatorname{dist}}

\def\id{\operatorname{id}}

\def\supp{\operatorname{supp}}

\newcommand{\overbar}[1]{\mkern 1.7mu\overline{\mkern-1.7mu#1\mkern-1.5mu}\mkern 1.5mu}

\def\mvint_#1{\mathchoice
          {\mathop{\vrule width 6pt height 3 pt depth -2.5pt
                  \kern -8pt \intop}\nolimits_{\kern -3pt #1}}%
          {\mathop{\vrule width 5pt height 3 pt depth -2.6pt
                  \kern -6pt \intop}\nolimits_{#1}}%
          {\mathop{\vrule width 5pt height 3 pt depth -2.6pt
                  \kern -6pt \intop}\nolimits_{#1}}%
          {\mathop{\vrule width 5pt height 3 pt depth -2.6pt
                  \kern -6pt \intop}\nolimits_{#1}}}

\numberwithin{theorem}{section} \numberwithin{equation}{section}

\title{Gluing diffeomorphisms, bi-Lipschitz mappings and homeomorphisms}
\author[P. Goldstein]{Pawe\l{}  Goldstein}
\address{Pawe\l{} Goldstein, Institute of Mathematics, Faculty of Mathematics, Informatics and Mechanics, University of Warsaw, Banacha 2, 02-097 Warsaw, Poland} \email{P.Goldstein@mimuw.edu.pl}
\thanks{P.G.\ was supported by NCN grant no 2019/35/B/ST1/02030}

\author[Z. Grochulska]{Zofia Grochulska}
\address{Zofia Grochulska, Department of Mathematics and Statistics, University of Jyväskylä, P.O. Box 35 (MaD), FI-40014 Finland\\ and  University of Warsaw, Banacha 2, 02-097 Warsaw, Poland} \email{zofia.z.grochulska@jyu.fi}

\thanks{Z.G.\ was supported by NCN grant no 2022/45/N/ST1/02977}

\author[P. Haj\l{}asz]{Piotr Haj\l{}asz}
\address{Piotr Hajlasz, Department of Mathematics, University of Pittsburgh, Pittsburgh, PA 15260, USA}
\email{hajlasz@pitt.edu}
\thanks{P.H.\ was supported by NSF grant DMS-2055171}

\begin{document}


\begin{abstract}
Cerf and Palais independently proved a remarkable result about extending diffeomorphisms defined on smooth balls in a manifold to global diffeomorphisms of the manifold onto itself. We explain Palais' argument and show how to extend it to the class of homeomorphisms and bi-Lipschitz homeomorphisms. While Palais' argument is surprising, it is elementary and short. However, its extension to bi-Lipschitz homeomorphisms and homeomorphisms requires deep results: the stable homeomorphism and the annulus theorems.
\end{abstract}
\maketitle

\section{Introduction}
The following question is very natural:
\begin{question}
\label{Q1}
Suppose that $F:B^n(0,1)\to\bbbr^n$ is a diffeomorphism onto its image. Assume also that $F$ can be extended to a diffeomorphism of some neighborhood of $\overbar{B}^n(0,1)$. Can we extend $F$ to a diffeomorphism of $\bbbr^n$ onto $\bbbr^n$?
\end{question}
Spoiler: The answer is ``yes'' and the proof is elementary! Take your time to think deeply about this problem. If you are not already familiar with a solution, you will likely find it {\em very} challenging. Good luck---you will need it!

The diffeomorphism $F$ can be very complicated and there is no apparent reason why an extension $F$ should exist:
it follows from the celebrated theorem of Milnor \cite[Theorem~5]{milnor} that if $A\subset\bbbr^7$ is a standard  annulus, then there is an orientation preserving surjective diffeomorphism $F:A\to A$ that can be extended to a diffeomorphism of a neighborhood of the closure $\overbar{A}$, can be extended to a bi-Lipschitz homeomorphism of $\bbbr^7$, but it cannot be extended to a diffeomorphism of $\bbbr^7$, see \Cref{T16}.

It appears that Cerf \cite{cerf2} and Palais \cite{Palais} were independently the first to address \Cref{Q1} and to provide a positive answer.
Palais' construction is remarkably short and elementary, yet so unintuitive that, even when one understands every step of the proof (which, as we mentioned, is simple), it remains difficult to grasp why it works.

In fact, Cerf and Palais proved a slightly more general result:
\begin{theorem}
\label{T9}
Let $\mathcal{M}^n$ be an $n$-dimensional connected and oriented manifold.  Assume that $B\subset\mathcal{M}^n$ is diffeomorphic to $B^n(0,1)$ and let $F:B\to \mathcal{M}^n$ be an orientation preserving diffeomorphism onto its image, that can be extended to a diffeomorphism of a neighborhood of the closure $\overbar{B}$. Then $F$ can be extended to a surjective diffeomorphism $F:\mathcal{M}^n\to\mathcal{M}^n$. 
\end{theorem}
While in \Cref{Q1} we do not require $F$ to be orientation preserving, we can always assume it, because if $F$ is orientation reversing, it will be orientation preserving after composing it with a reflection. However, in the case of manifolds the assumption about preservation of the orientation is necessary, because as pointed out by Palais \cite[Footnote~p.~276]{Palais}, there are oriented manifolds $\mathcal{M}^n$ such that every surjective homeomorphism $F:\mathcal{M}^n\to\mathcal{M}^n$ is orientation preserving. 

\Cref{T9} is so beautiful that it does not require any motivation. However, Palais \cite[Footnote~p.~275]{Palais} mentioned the following one: \Cref{T9} implies that the connected sum of oriented manifolds $\mathcal{M}^n$ and $\mathcal{N}^n$ is independent of the choice of balls whose boundaries we glue, because a ball in $\mathcal{M}^n$ (or in $\mathcal{N}^n$) can be transformed to another ball in $\mathcal{M}^n$ (or in $\mathcal{N}^n$) by a global diffeomorphism of the manifold $\mathcal{M}^n$ ($\mathcal{N}^n$). 

Proving \Cref{T9} boils down to answering \Cref{Q1}, let us sketch the answer here. As we pointed out, we can assume that $F$ is orientation preserving. By translation, we can assume that $F(0)=0$. Then, we can modify $F$ so that $F=\id$ in a neighborhood of $0$ and $F$ remains the original diffeomorphism near the boundary of $B^n(0,1)$, see \Cref{T15}. This \emph{local linearization} is standard and follows from the fact that in a small neighborhood of $0$, the diffeomorphism $F$ is very close to $DF(0)$, a linear isomorphism. Now, as was shown by Palais, we can define the extension of $F$ to a global diffeomorphism of $\bbbr^n$ by an explicit and elementary formula, see \Cref{T3}.

Here is why this result is unintuitive: The fact that $F=\id$ in a neighborhood of $0$ pertains to the interior of the ball, far from the boundary. Surprisingly, this information helps in constructing an extension of $F$, which could be perceived as a statement only about boundary values of $F$.

Having established the answer to \Cref{Q1}, one is prompted to ask about its analogue for less regular mappings:
\begin{question}
\label{Q2}
Suppose that $F:B^n(0,1)\to\bbbr^n$ is a homeomorphism (bi-Lipschitz homeomorphism) onto its image. Assume also that $F$ can be extended to a homeomorphism (bi-Lipschitz homeomorphism) of some neighborhood of $\overbar{B}^n(0,1)$. Can we extend $F$ to a homeomorphism (bi-Lipschitz homeomorphism) of $\bbbr^n$ onto $\bbbr^n$?
\end{question}
The answer is still ``yes''. Clearly, we can assume that $F(0)=0$, and the proof will be the same as before as long as we can modify $F$ so that $F=\id$ in a neighborhood of $0$. However, this time the problem of finding a local linearization of $F$, i.\,e., modifying $F$ so that $F=\id$ in a neighborhood of $0$ is {\em very} difficult and related to the deep stable homeomorphism and annulus theorems. These theorems are discussed later in this introduction.

The aim of this paper is to popularize the method of Palais and to extend it to the class of homeomorphisms and bi-Lipschitz homeomorphisms. While the Cerf-Palais Theorem \Cref{T9} is known to those who work with geometric topology, we believe that it should be known to those who use diffeomorphisms and are not necessarily very familiar with geometric topology. Moreover, the extensions of \Cref{T9} to the class of homeomorphisms and especially bi-Lipschitz homeomorphisms seem to be much less known. 

The main results of the paper are Theorems~\ref{T1},~\ref{T2} and~\ref{T8}. In \Cref{T1} we generalize \Cref{T9} to the case where instead of extending a diffeomorphism defined on one set diffeomorphic to a ball, we glue diffeomorphisms that are defined on many such sets. Then in Theorems~\ref{T2} and~\ref{T8} we extend \Cref{T1} to the cases of bi-Lipschitz homeomorphisms and homeomorphisms, respectively.

In order to state the results, we need to
fix notation. Throughout the paper we assume that $n\geq 2$. We denote by $\bbbn$ the set of all positive integers. The interior of a set $A$ will be denoted by $\mathring{A}$, the closure by $\bar{A}$. By a domain we mean an open and connected set. Surjective mappings will be denoted by $F:X\twoheadrightarrow Y$, i.\,e., $F(X)=Y$.
Open balls in the Euclidean $n$-space $\bbbr^n$ will be denoted by $B^n(x,r)$ or by $B(x,r)$ if the dimension is clear.
The unit ball and the unit sphere in $\mathbb{R}^n$ will be denoted by $\bbbb^n:=B^n(0,1)$ and $\bbbs^{n-1}:=\partial\bbbb^n$.

Let $\mathcal{M}^n$ and $\mathcal{N}^n$ be two $n$-dimensional manifolds and let $\Omega\subset \mathcal{M}^n$ be open. By a~homeomorphism $f:\Omega\to \mathcal{N}^n$ we mean a~homeomorphism onto the image, i.\,e., a~homeomorphism between $\Omega$ and $f(\Omega)$.
Similar convention applies to diffeomorphisms.

If $A\subset\mathcal{M}^n$ is closed, then by a $C^k$-diffeomorphism $F:A\to\mathcal{N}^n$ we mean a mapping that extends to a $C^k$-diffeomorphism on an open neighborhood of $A$. Finally, a closed set $D\subset\mathcal{M}^n$ is a {\em $C^k$-closed ball} if there is a $C^k$-diffeomorphism $F:\overbar{\bbbb}^n\twoheadrightarrow D\subset\mathcal{M}^n$. In other words, if there is a $C^k$-diffeomorphism $F:B^n(0,1+\eps)\to\mathcal{M}^n$ such that $F(\overbar{B}^n(0,1))=D$.

\begin{theorem}
\label{T1}
Let $\mathcal{M}^n$ be an $n$-dimensional connected and oriented manifold of class $C^k$, $k\in\bbbn\cup\{\infty\}$. Suppose that $\{{D}_i\}_{i=1}^\ell$ and $\{{D}'_i\}_{i=1}^\ell$, ${D}_i,{D}'_i\subset\mathcal{M}^n$, are two families of pairwise disjoint $C^k$-closed balls. If $F_i:{D}_i\twoheadrightarrow {D}'_i$, $i=1,2,\ldots,\ell$, are orientation preserving $C^k$-diffeomorphisms, then there is a $C^k$-diffeomorphism $F:\mathcal{M}^n\twoheadrightarrow\mathcal{M}^n$ such that $F|_{D_i}=F_i$. Moreover, if $D_i$ and $D'_i$ for all $i = 1, \ldots, \ell$ are contained in a domain $U \subset \mathcal{M}^n$, $F$ can be chosen to equal identity outside $U$.
\end{theorem}
\begin{remark}
We do not assume that the balls in the family $\{{D}_i\}_{i=1}^\ell$ are disjoint from the balls in the family $\{{D}'_i\}_{i=1}^\ell$.
\end{remark}

Let us stress that Theorem~\ref{T1} for $\ell=1$ coincides with the Cerf-Palais \Cref{T9}. In that case the result is about extending diffeomorphisms, while the case $\ell\geq 2$ is about gluing them. Our proof is based on Palais' argument (see Lemma~\ref{T3}).

Among the many possible corollaries based on Theorem~\ref{T1}, below we state one in the Euclidean setting, which is useful due to its elementary statement.
\begin{corollary}
\label{T13}
Let $\Omega\subset\bbbr^n$ be open and let $D_1$ and $D_2$ be $C^k$-closed balls, $k\in\bbbn\cup\{\infty\}$, such that $D_2\subset\mathring{D}_1\subset D_1\subset\Omega$. If $F:\Omega\to\bbbr^n$ and $G:D_2\to\bbbr^n$ are orientation preserving diffeomorphisms satisfying $G(D_2)\subset F(\mathring{D}_1)$, then there is a $C^k$-diffeomorphism $H:\Omega\twoheadrightarrow F(\Omega)$ that agrees with $F$ on $\Omega\setminus \mathring{D}_1$ and with $G$ on $D_2$.
\end{corollary}

Let $D_1=\overbar{B}^n(0,2)$ and $D_2=\overbar{B}^n(0,1)$. If $F:\bbbr^n\setminus\mathring{D}_1\twoheadrightarrow\bbbr^n\setminus \mathring{D}_1$ is an orientation preserving diffeomorphism and $G=\id$ on $D_2$, one could expect a similar conclusion as in Corollary~\ref{T13}: existence of a diffeomorphism $H:\bbbr^n\twoheadrightarrow\bbbr^n$ that agrees with $F$ on $\bbbr^n\setminus\mathring{D}_1$ and with $G$ on $D_2$. However, Milnor's theorem, \cite[Theorem~5]{milnor}, again yields a counterexample when $n=7$ (see Example~\ref{T16}).

We say that a closed set $D\subset\mathcal{M}^n$ is a {\em bi-Lipschitz closed ball}, if there is a bi-Lipschitz homeomorphism $F:\overbar{\bbbb}^n\twoheadrightarrow  {D}$. We say that $D$ is a  {\em flat bi-Lipschitz closed ball} if there is a bi-Lipschitz homeomorphism $F:\overbar{\bbbb}^n\twoheadrightarrow  {D}$ that can be extended to a bi-Lipschitz map on a neighborhood of $\overbar{\bbbb}^n$. Not every bi-Lipschitz ball is flat, as the example of the bi-Lipschitz Fox-Artin ball shows; see e.\,g. \cite{gehring}, \cite[Theorem~7.2.1]{GMP}, \cite[Theorem~3.7]{martin}.
For that reason in the next result we need to assume flatness of the balls.

\begin{theorem}
\label{T2}
Let $\mathcal{M}^n$ be an $n$-dimensional connected and oriented Lipschitz manifold.
Suppose that $\{{D}_i\}_{i=1}^\ell$ and $\{{D}'_i\}_{i=1}^\ell$, ${D}_i,{D}'_i\subset\mathcal{M}^n$, are two families of pairwise disjoint flat bi-Lipschitz closed balls. If $F_i:{D}_i\twoheadrightarrow {D}'_i$, $i=1,2,\ldots,\ell$,  are orientation preserving bi-Lipschitz homeomorphisms, then there is a bi-Lipschitz homeomorphism $F:\mathcal{M}^n \twoheadrightarrow \mathcal{M}^n$ such that $F|_{D_i}=F_i$. 
Moreover, if $D_i$ and $D'_i$ for all $i = 1, \ldots, \ell$ are contained in a domain $U \subset \mathcal{M}^n$, $F$ can be chosen to equal identity outside $U$.
\end{theorem}
\begin{remark}
Let us note that it follows from the Sch\"onflies theorem for Lipschitz maps \cite[Theorem A]{Tukia} that on the plane every bi-Lipschitz closed ball is flat, and hence the assumption on bi-Lipschitz flatness is superfluous for $n = 2$, see also \cite{DaneriPratelli, kovalev} for recent related results.
For a geometric characterization of flat bi-Lipschitz balls in $\bbbr^n$, $n\geq 2$, see Remark~\ref{R1}.
\end{remark}

While the overall strategy for Theorem~\ref{T2} is similar to that for Theorem~\ref{T1}, the problem is much deeper, because it requires 
the {\em bi-Lipschitz stable homeomorphism theorem} proved by Sullivan~\cite{Sullivan} 
(see Lemma~\ref{T26} below).
This is a very difficult result, but in dimensions $n\leq 3$ there is a simpler proof \cite{Tukia}, \cite[Theorem~3.2]{Vaisala77}. Unfortunately, the proof in the case $n\geq 4$ lacks details. However, Tukia and V\"ais\"al\"a  \cite{TukiaV} provided a detailed proof, except for the existence of Sullivan manifolds. Kapovich \cite{hajlaszMO}, \cite[Remark 12]{Kapovich} clarified how the existence of Sullivan's manifolds follows from Okun’s work \cite{okun}.

We say that a closed set $D\subset\mathcal{M}^n$ is a {\em topological closed ball}, if there is a homeomorphism $F:\overbar{\bbbb}^n\twoheadrightarrow  {D}$. We say that $D$ is a  {\em flat topological closed ball} if there is a homeomorphism $F:\overbar{\bbbb}^n\twoheadrightarrow  {D}$ that can be extended to a homeomorphism on a neighborhood of $\overbar{\bbbb}^n$ (see also Remark~\ref{R1}). Examples of wild spheres (the Alexander horned sphere) show that not every closed topological ball is flat, so, as in the bi-Lipschitz case, we need to assume flatness of balls.

\begin{theorem}
\label{T8}
Let $\mathcal{M}^n$ be an $n$-dimensional connected and oriented topological manifold.
Suppose that $\{{D}_i\}_{i=1}^\ell$ and $\{{D}'_i\}_{i=1}^\ell$, ${D}_i,{D}'_i\subset\mathcal{M}^n$, are two families of pairwise disjoint flat topological closed balls. If $F_i:{D}_i\twoheadrightarrow {D}'_i$, $i=1,2,\ldots,\ell$,  are orientation preserving homeomorphisms, then there is a homeomorphism $F:\mathcal{M}^n\twoheadrightarrow\mathcal{M}^n$ such that $F|_{D_i}=F_i$. Moreover, if $D_i$ and $D'_i$ for all $i = 1, \ldots, \ell$ are contained in a domain $U \subset \mathcal{M}^n$, $F$ can be chosen to equal identity outside $U$.
\end{theorem}
The proof of this result requires the stable homeomorphism theorem, see Lemma~\ref{T26}. For $n=2,3$ it was proved by Rad\'{o} \cite{Rado} and Moise \cite{Moise}, respectively; then for $n > 4$ by Kirby \cite{kirby} and, lastly, for $n=4$ by Quinn \cite{Quinn}.

In such topological problems as Theorems~\ref{T2} and \ref{T8}, the dimension $n=4$ is often doubtful, but the results are also true for $n=4$. 

Below we state an analogue of Corollary~\ref{T13} in the bi-Lipschitz and in the homeomorphic case.
\begin{corollary}
\label{T28}
Let $\Omega\subset\bbbr^n$ be open and let $D_1$ and $D_2$ be flat topological (or flat bi-Lipschitz) closed balls such that $D_2 \subset \mathring{D}_1 \subset D_1 \subset \Omega$. If $F:\Omega\to\bbbr^n$ and $G:D_2\to\bbbr^n$ are orientation preserving (bi-Lipschitz) homeomorphisms satisfying $G(D_2)\subset F(\mathring{D}_1)$, then there is a~(bi-Lipschitz) homeomorphism $H:\Omega\twoheadrightarrow F(\Omega)$ that agrees with $F$ on $\Omega\setminus \mathring{D}_1$ and with $G$ on $D_2$.
\end{corollary}

The paper is structured as follows. In Section~\ref{trick} we explain the beautiful argument of Palais \cite{Palais}. In Section~\ref{Proof1} we prove Theorem~\ref{T1} and Corollary~\ref{T13}.
The remaining sections are devoted to the proofs of Theorems~\ref{T2} and~\ref{T8}. As mentioned earlier, the main difficulty in these proofs lies in the generalization of the local linearization Lemma~\ref{T15}, which is easy only in the case of diffeomorphisms. For this, we need stable homeomorphism and annulus theorems. In Section~\ref{top} we collect known facts about these theorems in both cases, of homeomorphisms and bi-Lipschitz homeomorphisms. At the end of Section~\ref{top} we prove Lemma~\ref{T19}, which is a counterpart of Lemma~\ref{T15}. After that, the proofs of Theorems~\ref{T2} and~\ref{T8} are almost the same as the proof of Theorem~\ref{T1} and we discuss these proofs in Section~\ref{Proof23}.

\subsection*{Notation} Most of the notation is explained above. We shall also write $B(A, \eps):= \{x: \, \dist(x, A) < \eps\}$.

\subsection*{Acknowledgments} We would like to thank Igor Belegradek for directing us to the paper \cite{Palais} of Palais.

Piotr Hajłasz appreciates the hospitality of the University of Warsaw, where part of this work was conducted. His stay in Warsaw received funding from the
University of Warsaw via the IDUB project (Excellence Initiative Research University) as
part of the Thematic Research Programme \emph{Analysis and Geometry}.

\section{Palais}
\label{trick}

A well known result \cite[Lemma~8.1]{munkres60}, \cite[Theorem~5.5]{Palais1}, says that any orientation preserving $C^k$-diffeomorphism that fixes the origin, can be linearized to identity on a~sufficiently small ball.
\begin{lemma} 
\label{T15}
Suppose that $H:\overbar{B}(0, \varrho) \to \bbbr^n$ is an orientation preserving $C^k$-diffeomor\-phism, $k\in\bbbn\cup\{\infty\}$, with $H(0) = 0$. Then there is a constant $\delta \in (0, \varrho/2)$ and a~$C^k$-diffeomorphism $H_1: \overbar{B}(0, \varrho) \to \bbbr^n$,  such that $H_1 = H$ on $\overbar{B}(0, \varrho) \setminus B(0, \varrho/2)$ and $H_1 = \id$ on $B(0, \delta)$.
\end{lemma}
The idea is to glue $H$ with the linear mapping $DH(0)$ defined in a small neighborhood of $0$ using a suitable partition of unity and then to deform the linear map to the identity on a smaller neighborhood of $0$ using the fact that $DH(0)$ and the identity belong to the same connected component of $GL(n)$. One also needs some topology to ensure that the resulting mapping is a diffeomorphism and not only a local diffeomorphism. For a~detailed proof, we refer to Lemma~3.8 and Lemma~3.10 from \cite{GGH}.

The next surprising result was proved by Palais \cite{Palais}. 
It is the core of the proof of Theorem~\ref{T1}.
\begin{lemma}
\label{T3}
Suppose that $H:\overbar{B}(0, \varrho) \to \bbbr^n$ is an orientation preserving $C^k$-diffeomor\-phism, $k\in\bbbn\cup\{\infty\}$, with $H(0) = 0$. Then for any $\eps > 0$, there is a $C^k$-diffeomorphism $\widetilde{H}:\bbbr^n \twoheadrightarrow \bbbr^n$ such that
\begin{equation}
\label{eq1}
\widetilde{H}(x)=
\begin{cases} 
H(x) &\text{ if } x\in\overbar{B}(0,\varrho),\\
x &\text{ if } \dist(x,A)\geq\eps,
\end{cases}
\end{equation}
where $A=\overbar{B}(0,\varrho)\cup H(\overbar{B}(0,\varrho))$.
\end{lemma}
(Note that the formula \eqref{eq1} describes a set of conditions satisfied by $\widetilde{H}$, not a definition for $\widetilde{H}$; we shall use this notation repeatedly in the paper.)
\begin{remark}
If $\widetilde{H}:\bbbr^n\twoheadrightarrow\bbbr^n$ is a diffeomorphism such that $\widetilde{H}(0)=0$, but $\widetilde{H}(x)\neq x$ for $0\neq x\in\overbar{B}(0,\varrho)$, then clearly $\widetilde{H}(x)\neq x$ for $0\neq x\in A$ and hence $\widetilde{H}(x)\neq x$ in a neighborhood of $\partial A$. Therefore, the condition $\widetilde{H}(x)=x$ if $\dist(x,A)\geq\eps$ in \eqref{eq1} is sharp.
\end{remark}
\begin{proof}
Let $\eps>0$ be given. By assumption, $H$ extends to a $C^k$-diffeomorphism on $B(0, \varrho + 3 \tau)$ for some $\tau > 0$; we denote this extension by $H$ as well. By decreasing $\tau>0$ if necessary, we may assume that
$$
B(0,\varrho+3\tau)\cup H(B(0,\varrho+3\tau))\subset B(A, \eps).
$$
By Lemma~\ref{T15}, there is a~constant $\delta\in (0,\varrho/2)$ and a $C^k$-diffeomorphism $H_1:B(0,\varrho+3\tau)\to\bbbr^n$ such that
$$
H_1(x)=
\begin{cases}
H(x) & \text{ if } x\in B(0,\varrho+3\tau)\setminus B(0,\varrho/2),\\
x    & \text{ if } x\in B(0,\delta).
\end{cases}
$$

Let $\phi:\bbbr\to\bbbr$ be a non-decreasing, smooth function satisfying 
$$
\phi(t)=
\begin{cases} 
1 &\text{ if } t>\varrho+2\tau,\\
\delta(\varrho+\tau)^{-1} & \text{ if } t<\varrho+\tau,
\end{cases}
$$
and define $\Phi:\bbbr^n \twoheadrightarrow \bbbr^n$ by $\Phi(x):=\phi(|x|)x$. The mapping $\Phi$ is obviously a smooth diffeomorphism of $\bbbr^n$. Note also that $\Phi$ acts as scaling by a factor of $\delta(\varrho+\tau)^{-1}$ on $B(0,\varrho+\tau)$ and that $\Phi(B(0,\varrho+\tau))=B(0,\delta)$. Moreover, $\Phi(x)=x$ when $|x|>\varrho+2\tau$. 
Since $B(0, \varrho + 2\tau) \subset B(A, \eps)$, $\Phi = \id$ in $\bbbr^n \setminus B(A, \eps)$.

Consider the $C^k$-diffeomorphism
\begin{equation}
\label{eq2}
H_1\circ\Phi^{-1}\circ H_1^{-1}: H_1(B(0,\varrho+3\tau))\to\bbbr^n.
\end{equation}
It is a well defined diffeomorphism, because $\Phi$ maps the ball $B(0,\varrho+3\tau)$ onto itself. The diffeomorphism defined in~\eqref{eq2} is identity near the boundary of $H_1(B(0,\varrho+3\tau))$, because
\begin{equation}
\label{eq3}
H_1\circ \Phi^{-1}\circ H_1^{-1}(x)=x
\quad
\text{for}
\quad
x\in H_1\big(B(0,\varrho+3\tau)\setminus B(0,\varrho+2\tau)\big).
\end{equation}
Indeed, $\Phi^{-1}$ is the identity on $B(0,\varrho+3\tau)\setminus B(0,\varrho+2\tau)$, so for $x$ as in \eqref{eq3}, we have $\Phi^{-1}(H_1^{-1}(x))=H_1^{-1}(x)$.
Therefore, the $C^k$-diffeomorphism \eqref{eq2} has the extension to a $C^k$-diffeomorphism of $\bbbr^n$ by identity:
\begin{equation}
\label{eq4}
\reallywidetilde{H_1\circ \Phi^{-1}\circ H_1^{-1}}=
\begin{cases}
H_1\circ \Phi^{-1}\circ H_1^{-1} & \text{ in } H_1(B(0,\varrho+3\tau)),\\
\id                              & \text{ in } \bbbr^n\setminus H_1(B(0,\varrho+3\tau)).
\end{cases}
\end{equation}
Moreover, the mapping defined in~\eqref{eq4} is identity in $\bbbr^n \setminus B(A, \eps)$, because $H_1(B(0,\varrho+3\tau))=H(B(0,\varrho+3\tau))\subset B(A, \eps)$.

Now, we define the $C^k$-diffeomorphism
\begin{equation}
\label{eq6}
H_2=\big(\reallywidetilde{H_1\circ \Phi^{-1}\circ H_1^{-1}}\big)\circ\Phi:\bbbr^n\twoheadrightarrow\bbbr^n.
\end{equation}
Note that since $\Phi$ and the diffeomorphism defined in~\eqref{eq4} are both equal to identity in the set $\bbbr^n \setminus B(A, \eps)$, so is $H_2$.

If $x\in B(0,\varrho+\tau)$, then
$$
\Phi(x)=\delta(\varrho+\tau)^{-1}x\in B(0,\delta)\subset H_1(B(0,\varrho+3\tau)).
$$
Thus by \eqref{eq4}, and taking into account that  $H_1^{-1}=\id$ on $B(0,\delta)$, for $x\in B(0,\varrho+\tau)$ we have
$$
H_2(x)=H_1\circ\Phi^{-1}\circ H_1^{-1}(\Phi(x))=
H_1\circ \Phi^{-1}\circ \Phi(x)=H_1(x).
$$

Finally, since $H_2=H_1=H$ near $\partial B(0,\varrho)$, we can set
$$
\widetilde{H}:=
\begin{cases} 
H_2&\text{ on }\bbbr^n\setminus \overbar{B}(0,\varrho),\\
H&\text{ on }\overbar{B}(0,\varrho),
\end{cases}
$$
and, clearly, $\widetilde{H}$ is a diffeomorphism that satisfies \eqref{eq1}.
\end{proof}

\begin{corollary}
\label{T7}
Let $G: \overbar{B}(a, r) \twoheadrightarrow \overbar{B}(a, r)$ be an orientation preserving $C^k$-diffeomorphism, $k\in\bbbn\cup\{\infty\}$. Then for any $\eps > 0$, there is a $C^k$-diffeomorphism $\widetilde{G}: \bbbr^n \twoheadrightarrow \bbbr^n$ such that $\widetilde{G} = G$ on $\overbar{B}(a, r)$ and $\widetilde{G} = \id$ outside $B(a, r + \eps)$.
\end{corollary}
\begin{proof}
If $G(a) = a$, it is a straightforward consequence of Lemma~\ref{T3}. If not, we construct a $C^\infty$-diffeomorphism $F: \bbbr^n \twoheadrightarrow \bbbr^n$ such that $F(G(a)) = a$ and $F = \id$ in $\bbbr^n\setminus B(a, r)$. 
One can construct $F$ through a $1$-parameter group of diffeomorphisms generated by a compactly supported vector field, see e.g. \cite[Lemma 3.13]{GGH}.

Let $T(x) = x-a$ and $G_1 = T \circ F \circ G \circ T^{-1}$. Then $G_1$ is an orientation preserving  $C^k$-diffeomorphism of $\overbar{B}(0,r)$ onto itself, $G_1(0)=0$. For any $\eps > 0$, we may thus apply Lemma~\ref{T3} to $G_1$, obtaining a diffeomorphism $\widetilde{G}_1:\bbbr^n\twoheadrightarrow\bbbr^n$ such that 
$$
\widetilde{G}_1(x)=
\begin{cases}
    G_1(x) &\text{ for }x\in \overbar{B}(0,r),\\
    x & \text{ for }x \not \in B(0,r+\eps).
\end{cases}
$$
Now, $\widetilde{G}=F^{-1}\circ T^{-1}\circ \widetilde{G}_1\circ T$ satisfies the claim of the lemma.
\end{proof}

\section{Proof of Theorem \ref{T1}}
\label{Proof1}

\begin{lemma}
\label{T4}
Let $\{ D_i\}_{i=1}^\ell$, $D_i\subset\bbbr^n$, be a family of pairwise disjoint $C^k$-closed balls and let $p_i\in\mathring{D}_i$, $i=1,\ldots,\ell$, be given. Then there is $\eps_o >0$ such that for any $\eps\in (0,\eps_o)$ there is a $C^k$-diffeomorphism $F_\eps:\bbbr^n\twoheadrightarrow\bbbr^n$ satisfying $F_\eps(\overbar{B}(p_i,\eps))=D_i$ for $i=1,\ldots,\ell$ and 
$$
F_\eps(x)=x
\text{ if } \dist\Big(x,\bigcup \nolimits_{i=1}^\ell D_i\Big)\geq\eps.
$$ 
\end{lemma}
\begin{proof}
Let $0<\eps_o\leq \frac{1}{4}\min_{i\neq j}\dist(D_i,D_j)$ be such that $\overbar{B}(p_i,\eps_o)\subset \mathring{D}_i$ for all $i=1,2,\ldots,\ell$ and let $\eps\in (0,\eps_o)$. By assumption, there are $C^k$-diffeomorphisms $H_i:\overbar{B}(p_i,\eps)\twoheadrightarrow D_i$. Additionally, we can assume that $H_i(p_i)=p_i$. Indeed, 
if $H_i(p_i)\neq p_i$, we compose $H_i$ with a diffeomorphism of $D_i$ which is identity near $\partial D_i$ and maps $H_i(p_i)$ to $p_i$. Such a diffeomorphism can be constructed through a $1$-parameter group of diffeomorphisms generated by a compactly supported vector field. Note that
$$
\overbar{B}(p_i,\eps)\cup H_i(\overbar{B}(p_i,\eps))=D_i
\quad
\text{for } i=1,2,\ldots,\ell.
$$
According to Lemma~\ref{T3}, each $H_i$ can be extended to a diffeomorphism $\widetilde{H}_i:\bbbr^n\twoheadrightarrow\bbbr^n$ that is identity outside the $\eps$-neighborhood of $D_i$. Since the closures of the sets $B(D_i, \eps)$ are pairwise disjoint, we can glue the diffeomorphisms $\widetilde{H}_i$ to a diffeomorphism $F_\eps$ satisfying the claim of the lemma.
\end{proof}

The next result is well known.
The diffeomorphism $H$ can be constructed using $1$\nobreakdash-para\-meter groups of diffeomorphisms with compact support in $U$, see e.g. \cite[Lemma 3.13]{GGH}.
\begin{lemma}
\label{T5}
Let $U\subset\bbbr^n$ be a domain and let
$\{p_i\}_{i=1}^\ell$ and $\{q_i\}_{i=1}^\ell$ be given points in $U$such that $p_i \neq p_j$ and $q_i \neq q_j$ for $i \neq j$. Then, there exists an~$\eps_o > 0$ such that for any $\eps\in (0,\eps_o)$ there is a~$C^\infty$-diffeomorphism $H: \bbbr^n\twoheadrightarrow\bbbr^n$, identity on $\bbbr^n\setminus U$, such that
\begin{equation}
\label{eq T5}
H(\overbar{B}(q_i,\eps))= \overbar{B}(p_i,\eps).
\end{equation}
\end{lemma}

Combining Lemmata~\ref{T4} and~\ref{T5} we obtain the following result, which is a Euclidean variant of Theorem~\ref{T1}.
\begin{corollary}
\label{T6}
Let $U\subset\bbbr^n$ be a domain.
Suppose that $\{D_i\}_{i=1}^\ell$ and $\{D'_i\}_{i=1}^\ell$, $D_i, D'_i\subset U$, are two families of pairwise disjoint $C^k$-closed balls. If $F_i:D_i\twoheadrightarrow D'_i$, $i=1,2,\ldots,\ell$, are orientation preserving $C^k$-diffeomorphisms, then there is a $C^k$-diffeomorphism ${F:\bbbr^n\twoheadrightarrow\bbbr^n}$ such that $F|_{D_i}=F_i$ and $F=\id$ in $\bbbr^n\setminus U$.
\end{corollary}
\begin{proof}
Let $p_i\in\mathring{D}_i$ and $q_i\in\mathring{D}_i'$. By Lemma~\ref{T4}, there are diffeomorphisms ${A_1,A_2:\bbbr^n\twoheadrightarrow\bbbr^n}$ such that
$$
A_1(\overbar{B}(p_i,\eps))=D_i, 
\quad
A_2(\overbar{B}(q_i,\eps))=D_i',
\quad
\text{and}
\quad
A_{i}=\id \text{ in } \bbbr^n\setminus U \quad \text{for } i = 1,2.
$$
This is true provided $\eps>0$ is sufficiently small. In what follows we will not clarify how small $\eps$ should be, but it will always be clear from the context.

Let $H:\bbbr^n\twoheadrightarrow\bbbr^n$ be a diffeomorphism as in Lemma~\ref{T5}. Then,
$$
G_i:=H\circ A_2^{-1}\circ F_i\circ A_1:\overbar{B}(p_i,\eps) \twoheadrightarrow \overbar{B}(p_i,\eps),
\quad
i=1,2,\ldots, \ell
$$
are orientation preserving diffeomorphisms. According to Corollary~\ref{T7} each of them can be extended to a diffeomorphism $\widetilde{G}_i:\bbbr^n\twoheadrightarrow\bbbr^n$ that is identity outside an arbitrarily small neighborhood of $\overbar{B}(p_i,\eps)$ so they can be glued to a diffeomorphism $G:\bbbr^n\twoheadrightarrow\bbbr^n$ satisfying $G|_{\overbar{B}(p_i,\eps)}=G_i$.
We can clearly guarantee that $G=\id$ in $\bbbr^n\setminus U$. Then, it easily follows that
$F:=A_2\circ H^{-1}\circ G\circ A_1^{-1}$ satisfies the claim of the corollary.
\end{proof}

\begin{proof}[Proof of Corollary~\ref{T13}]
Since $F^{-1}\circ G:D_2\to\mathring{D}_1$ is an orientation preserving diffeomorphism between $C^k$-closed balls
$D_2$ and $(F^{-1}\circ G)(D_2)\subset \mathring{D}_1$, it follows from Corollary~\ref{T6} that there is a diffeomorphism $\Phi:\bbbr^n\twoheadrightarrow\bbbr^n$ such that $\Phi=F^{-1}\circ G$ in $D_2$ and $\Phi=\id$ in $\bbbr^n\setminus\mathring{D}_1$. Then the diffeomorphism $H:=F\circ\Phi:\Omega\to\bbbr^n$ is well defined and satisfies $H|_{D_2}=G$, $H|_{\Omega\setminus\mathring{D}_1}=F$.
\end{proof}

\begin{proof}[Proof of Theorem~\ref{T1}.]
We prove the result by reducing it to the Euclidean setting, i.e., to the case treated in Corollary~\ref{T6}. 

Let $U$ be a domain containing $D_i, D'_i$ for $i = 1, \ldots, \ell$; note that we include the case $U=\mathcal{M}^n$. Let $K\subset U$ be a $C^k$-closed ball. Then, there are $C^k$-diffeomorphisms $H,H':\mathcal{M}^n\twoheadrightarrow\mathcal{M}^n$ such that $H(D_i)\subset\mathring{K}$, $H'(D_i')\subset\mathring{K}$ for all $i=1,2,\ldots,\ell$, and $H$ and $H'$ are identity outside $U$.

We will sketch the construction of $H$ leaving details to the reader. The construction of $H'$ is analogous. Since $D_i$ is diffeomorphic to $\overbar{\bbbb}^n$, we can find a diffeomorphism $\Psi_i:\mathcal{M}^n\twoheadrightarrow\mathcal{M}^n$, identity outside a small neighborhood of $D_i$, that maps $D_i$ into a small neighborhood of $p_i\in\mathring{D}_i$ (cf.\ the construction of $\Phi$ in the proof of Lemma~\ref{T3}). Since the sets $D_i$ are pairwise disjoint, we can glue the diffeomorphisms $\Psi_i$ to $\Psi:\mathcal{M}^n\twoheadrightarrow\mathcal{M}$ that is identity outside a small neighborhood of $\bigcup_{i=1}^\ell D_i$, and hence identity outside $U$.

Next, we construct a vector field $X$ with compact support in $U$, whose $1$-parameter group of diffeomorphisms $\Phi_t$ satisfies $\Phi_{t_o}(\Psi(D_i))\subset\mathring{K}$ for all $i$, provided $t_o>0$ is sufficiently large. To construct such a vector field, we take smooth arcs $\gamma_i$ in $U$, connecting $p_i$ to some $p_i'\in\mathring{K}$ and making sure that $\gamma_i\cap \gamma_j=\varnothing$ if $i\neq j$. Then, we take vector fields along $\gamma_i$ in the direction from $p_i$ to $p_i'$, and extend them to compactly supported vector fields $X_i$ on $\mathcal{M}^n$ with the supports near $\gamma_i$. Clearly, we can guarantee that $\supp X_i\cap\supp X_j=\varnothing$ if $i\neq j$, 
and $\supp X_i\subset U$ for all $i$,
so the vector fields glue to a compactly supported vector field $X$ on $\mathcal{M}^n$ with $\supp X\subset U$. Let $\Phi_t$ be the 1-parameter group of diffeomorphisms generated by $X$. If $\Psi(D_i)$ lies in a sufficiently small neighborhood of $p_i$ (which we can guarantee), then $\Phi_{t_o}(\Psi(D_i))\subset\mathring{K}$ for all $i$ and sufficiently large $t_o>0$. Thus, we can take $H:=\Phi_{t_o}\circ\Psi$. Clearly, $H=\id$ in $\mathcal{M}^n\setminus U$.

Let  
$\widehat{F}_i:H(D_i)\twoheadrightarrow H'(D_i')$,
$\widehat{F}_i=H'\circ F_i\circ H^{-1}$.
If we represent $K$ in a coordinate system, Corollary~\ref{T6} allows us to construct a diffeomorphism $\widehat{F}:K\twoheadrightarrow K$ such that
$$
\widehat{F}|_{H(D_i)}=\widehat{F}_i
\quad
\text{and}
\quad
\widehat{F}=\id \text{ near } \partial K,
$$
so $\widehat{F}$ extends to a diffeomorphism $\widehat{F}:\mathcal{M}^n\twoheadrightarrow\mathcal{M}^n$ by identity outside $K$. Now, it easily follows that
$$
F:=(H')^{-1}\circ\widehat{F}\circ H:\mathcal{M}^n\twoheadrightarrow\mathcal{M}^n
$$
satisfies $F|_{D_i}=F_i$
and $F=\id$ in 
$\mathcal{M}^n\setminus U$.
\end{proof}

\begin{example}
\label{T16}
In his celebrated article \cite{milnor} introducing exotic $7$-spheres, Milnor proved the existence of an orientation preserving $C^\infty$-diffeomorphism $f: \bbbs^6 \twoheadrightarrow \bbbs^6$ which is not isotopic through diffeomorphisms to identity (\cite[Theorem 5]{milnor}). This diffeomorphism can be extended radially to
a bi-Lipschitz homeomorphism of $\bbbr^7$, such that $\bbbr^7 \setminus \{0\}\to \bbbr^7$ is a $C^\infty$-diffeomorphism. 
However, it cannot be extended to a $C^\infty$-diffeomorphism onto $\bbbr^7 \twoheadrightarrow \bbbr^7$. If it were, it would be possible to find a $C^\infty$-diffeomorphism $G$ which coincides with $f$ on $\bbbs^6$ and which is equal to identity on a small ball $B(0,r)$, like in Lemma~\ref{T15}.
Since the annulus $\overbar{\bbbb}^7\setminus B(0,r)$ is diffeomorphic to $\bbbs^6\times [0,1]$, we would have
a diffeomorphism $F:\mathbb{S}^6\times[0,1]\twoheadrightarrow \mathbb{S}^6\times[0,1]$, $F(x,1)=f(x)$, $F(x,0)=x$. Such a diffeomorphism need not be an isotopy (the latter must preserve the sets $\mathbb{S}^6\times\{t\}$), but it is a pseudoisotopy between $f$ and the identity, and since we discuss diffeomorphisms of $\mathbb{S}^6$, by Cerf's pseudoisotopy-to-isotopy theorem (\cite{cerf}, see also  \cite[Theorem 2.71]{juhasz}) this implies that $f$ is isotopic to the identity, which is a contradiction.  

Also, if we consider $\bbbs^6$ to be embedded into $\bbbs^7$, we see that the same $f$ can be extended to a $C^\infty$-homeomorphism on an annulus $A \subset \bbbs^7$, but cannot be extended to a $C^\infty$-diffeomorphism of $\bbbs^7$.
\end{example}

\section{Topological local linearization}
\label{top}

\begin{definition}
Let $\varrho > 0$. We say that a homeomorphism $F: \bbbr^n \twoheadrightarrow \bbbr^n$ (or $F: B(0, \varrho) \twoheadrightarrow B(0, \varrho)$) is {\em stable} if
\begin{equation} \label{eq8}
F = f_1 \circ \ldots \circ f_k
\end{equation}
for some $k \in \mathbb{N}$ and homeomorphisms $f_i: \bbbr^n \twoheadrightarrow \bbbr^n$ (resp., $f_i: B(0,\varrho) \twoheadrightarrow B(0, \varrho)$) such that $f_i|_{U_i} = \id$ for some nonempty open set $U_i \subset \bbbr^n$ (resp., $U_i \subset B(0, \varrho)$).

We say that a homeomorphism $F: \bbbr^n \twoheadrightarrow \bbbr^n$ (or $F: B(0, \varrho) \twoheadrightarrow B(0, \varrho)$) is {\em locally bi-Lipschitz stable} if $F$ can be written as in \eqref{eq8} for some $k \in \mathbb{N}$ and for locally bi-Lipschitz homeomorphisms $f_i: \bbbr^n \twoheadrightarrow \bbbr^n$ (resp., $f_i: B(0,\varrho) \twoheadrightarrow B(0, \varrho)$) such that $f_i|_{U_i} = \id$ for some nonempty open set $U_i \subset \bbbr^n$ (resp., $U_i \subset B(0, \varrho)$).
\end{definition}
\begin{remark}
Without loss of generality, we may assume that the $U_i$ in the definition above are pairwise disjoint balls with the same radius (compactly contained in $B(0, \varrho)$).
\end{remark}
\begin{lemma}[Stable homeomorphism theorem] 
\label{T26}
Any orientation preserving homeomorphism $F: \bbbr^n \twoheadrightarrow \bbbr^n$ is stable. If $F$ is additionally locally bi-Lipschitz, then $F$ is locally bi-Lipschitz stable.
\end{lemma}

The stable homeomorphism theorem in the homeomorphic case in dimensions $n = 2, 3$ is a classical result due to Rad\'{o} \cite{Rado} for $n = 2$ and Moise \cite{Moise} for $n=3$. For $n > 4$, it was proved by Kirby in \cite{kirby} and for $n = 4$ by Quinn in \cite{Quinn} (see page 1 and Theorem 2.2.2). For a short explanation of the intricacies of the proof in the topological case, we refer to \cite{hatcherMO}. In the bi-Lipschitz case, the stable homeomorphism theorem was derived in \cite[Theorem 3.12]{TukiaV} from the existence of groups known as Sullivan groups, whose existence was claimed by Sullivan in \cite{Sullivan}.

Since the open unit ball is homeomorphic to $\bbbr^n$ and this homeomorphism can be chosen to be locally bi-Lipschitz, the following is easy to prove.
\begin{corollary} 
\label{T17}
An orientation preserving homeomorphism $F: B(0, \varrho) \twoheadrightarrow B(0, \varrho)$ is stable. If, additionally, $F$ is locally bi-Lipschitz, then $F$ is also locally bi-Lipschitz stable.
\end{corollary}
The next result is a version of Lemma~\ref{T15} in the case of homeomorphisms and bi-Lipschitz mappings. However, it is not a good one, because the assumption that $F$ maps $B(0, \varrho)$ onto itself is very restrictive. We will remove this assumption later in Lemma~\ref{T19} using the annulus theorem.
\begin{lemma} 
\label{T18}
Let $F: B(0, \varrho) \twoheadrightarrow B(0, \varrho)$ be an orientation preserving homeomorphism. Then for any $\delta \in (0, \varrho)$, there is $\tau\in (\delta,\varrho)$ and a homeomorphism $\widehat{F}: B(0, \varrho) \twoheadrightarrow B(0, \varrho)$ such that
\begin{equation} 
\label{eq7}
\widehat{F}(x) = \
\begin{cases}
F(x) & \text{for } x \in B(0, \varrho)\setminus B(0,\tau),\\
x & \text{ for } x \in B(0, \delta).
\end{cases}
\end{equation}
If $F$ is locally bi-Lipschitz, then $\widehat{F}$ is locally bi-Lipschitz as well.
\end{lemma}
\begin{proof}
By Corollary~\ref{T17}, there are homeomorphisms $f_i: B(0, \varrho) \twoheadrightarrow B(0, \varrho)$, $i = 1, \ldots,k$, and balls $\overbar{U}_i\subset B(0,\varrho)$ such that
\begin{equation} 
\label{3eq7}
F = f_1 \circ \ldots \circ f_k, \quad f_i|_{U_i} = \id.
\end{equation}
Let $\lambda\in (\delta,\varrho)$ be such that $\bigcup_{i=1}^k\overbar{U}_i\subset B(0,\lambda)$. 
Since $f_i$ are surjective homeomorphisms of $B(0, \varrho)$, we have $|f_i(x)| \to \varrho$ as $|x| \to \varrho$, so there is $\tau\in(\lambda,\varrho)$ such that 
\begin{equation}
\label{eq9}
|f_k(x)|>\lambda,
\quad
|(f_{k-1}\circ f_k)(x)|>\lambda,
\ldots,
|(f_1\circ\ldots\circ f_k)(x)|>\lambda
\quad
\text{whenever } |x|>\tau.
\end{equation}
Let $\psi_i:B(0,\varrho)\twoheadrightarrow B(0,\varrho)$ be a diffeomorphism that maps $B(0,\delta)$ onto $U_i$ and equals identity on $B(0,\varrho)\setminus B(0,\lambda)$. Then
$\psi_i^{-1}\circ f_i\circ\psi_i:B(0,\varrho)\twoheadrightarrow B(0,\varrho)$ is a homeomorphism which coincides with identity on $B(0,\delta)$. We set
$$
\widehat{F} := \psi_1^{-1} \circ f_1 \circ \psi_1 \, \circ \, \psi^{-1}_2 \circ f_2 \circ \psi_2 \,  \circ \, \ldots \circ \,  \psi_k^{-1} \circ f_k \circ \psi_k.
$$
Clearly, $\widehat{F}$ satisfies the second condition in \eqref{eq7}, and \eqref{eq9} along with a simple induction imply that $\widehat{F}$ satisfies the first condition in \eqref{eq7}.
Moreover, if $F$ is locally bi-Lipschitz, we can assume that $f_i$ are locally bi-Lipschitz and hence $\widehat{F}$ is locally bi-Lipschitz.
\end{proof}

In accordance with our definitions of flat topological and bi-Lipschitz balls, we say that $S\subset\bbbr^n$ is a \emph{flat topological} (or flat bi-Lipschitz) \emph{sphere}, if there is a (bi-Lipschitz) homeomorphism $F:\bbbs^{n-1}\times(-1,1)\to\bbbr^n$ such that $F(\bbbs^{n-1}\times\{0\})=S$. 

\begin{remark} 
\label{R1}
There is an elegant characterization of flat topological (or flat bi-Lipschitz) spheres and balls in $\bbbr^n$. 
\begin{itemize}
    \item[a)] A topological (or bi-Lipschitz) sphere $S\subset \bbbr^n$ is flat if and only if it is a (bi-Lipschitz) submanifold of $\bbbr^n$, that is every point $x\in S$ has a neighborhood $U$ such that the pair $(U,S\cap U)$ is (bi-Lipschitz) homeomorphic to ${(\bbbb^n, \{(x_1,\ldots,x_n)\in \bbbb^n} \,\,:\,\,x_n=0\})$.
    \item[b)] A topological (or bi-Lipschitz) closed ball $D\subset\bbbr^n$ is flat if and only if $\partial D$ is a flat topological (or flat bi-Lipschitz) sphere.
\end{itemize}
\end{remark}
Obviously, a flat sphere is a submanifold, both in the topological and the bi-Lipschitz case. The reverse assertion in a) follows from the seminal paper of Brown \cite[Lemma 6 and Theorem 4]{brown2} (note that a flat sphere as per our definition is a bi-collared sphere in Brown's paper), see also \cite{connely}. The bi-Lipschitz case is covered in \cite[Theorem 7.8]{LuukkainenVaisala}.
Assertion b) follows from the generalized Sch\"onflies theorem proved by Brown \cite[Theorem 5]{brown}; the bi-Lipschitz case again follows immediately from \cite[Theorem 7.8]{LuukkainenVaisala}.
\begin{lemma}[Annulus theorem] 
\label{T20}
Let $S_1, S_2 \subset \bbbr^n$ be two disjoint flat topological (or flat bi-Lipschitz) spheres such that $S_1$ is contained in the bounded component of $\bbbr^n \setminus S_2$. Then the set bounded by $S_1$ and $S_2$ is (bi-Lipschitz) homeomorphic to the annulus $\overbar{\bbbb}^n \setminus B^n(0, 1/2)$.
\end{lemma}
Brown and Gluck \cite{BrownGluck} proved that the annulus theorem (in the case of flat topological spheres) is a direct consequence of the stable homeomorphism theorem (in the case of homeomorphisms) and the generalized Sch\"onflies theorem \cite[Theorem 5]{brown}. See Corollary on page 8 and Theorem~3.5(i) in \cite{BrownGluck}. Their argument is elementary and geometric (but very tricky) and it verbatim applies to the bi-Lipschitz case: the bi-Lipschitz stable homeomorphism theorem implies the bi-Lipschitz annulus theorem, see \cite[Theorem 3.12]{TukiaV}.

\begin{remark} 
\label{T30} 
By a~radial extension of a~homeomorphism $f: \partial B(0, \varrho) \twoheadrightarrow \partial B(0, \varrho)$ we mean a mapping $\tilde{f}: \bbbr^n \twoheadrightarrow \bbbr^n$ defined as $\tilde{f}(x) = \varrho^{-1}|x| f(\varrho x/|x|)$ for $x \neq 0$ and $\tilde{f}(0) = 0$. Clearly, $\tilde{f}$ is a~homeomorphism and if $f$ is bi-Lipschitz, then $\tilde{f}$ is bi-Lipschitz.
\end{remark}
The next result, \cite[Theorem 3.16]{TukiaV}, is a generalization of Lemma~\ref{T15} to the case of homeomorphisms and bi-Lipschitz maps. It will allow us to use Palais' argument (Lemma~\ref{T3}) beyond the realm of diffeomorphisms (Lemma~\ref{T21}).
\begin{lemma} 
\label{T19}
Let $F: \overbar{B}(0, \varrho) \to \bbbr^n$, $F(0) = 0$, be an orientation preserving (bi-Lipschitz) homeomorphism. Then, there is $\delta \in (0, \varrho/2)$  and a (bi-Lipschitz) homeomorphism $\widetilde{F}: \overbar{B}(0, \varrho) \to \bbbr^n$ such that $\widetilde{F} = F$ on $\overbar{B}(0, \varrho)\setminus B(0,\varrho/2)$ and $F = \id$ on $\overbar{B}(0, \delta)$. 
\end{lemma}
\begin{proof}
Firstly, note that since $F(0) = 0$, there is  $\delta \in (0, \varrho/4)$ for which $\overbar{B}(0, 2\delta) \subset F(B(0, \varrho/2))$. 

The set $\partial F(B(0, \varrho/2)) = F(\partial B(0, \varrho/2))$ is a flat topological sphere; let $A$ denote the compact region between $\partial F(B(0, \varrho/2))$ and $\partial B(0, 2\delta)$. By Lemma~\ref{T20}, there is a homeomorphism $H: \overbar{B}(0, \varrho/2) \setminus B(0, 2\delta) \twoheadrightarrow A$ and therefore
$$
H^{-1} \circ F: \partial B(0, \varrho/2) \twoheadrightarrow \partial B(0, \varrho/2) \text{ and } H: \partial B(0, 2\delta) \twoheadrightarrow \partial B(0, 2\delta).
$$
We can find radial extensions of $H^{-1} \circ F$ and $H$ to homeomorphisms $\Phi: \overbar{B}(0, \varrho/2) \twoheadrightarrow \overbar{B}(0, \varrho/2)$ and $\Psi: \overbar{B}(0, 2\delta) \twoheadrightarrow \overbar{B}(0, 2\delta)$, respectively. By Lemma~\ref{T18}, we find homeomorphisms $\widetilde{\Phi}: \overbar{B}(0, \varrho/2) \twoheadrightarrow \overbar{B}(0, \varrho/2)$ and $\widetilde{\Psi}: B(0, 2\delta) \twoheadrightarrow B(0, 2\delta)$ such that $\widetilde{\Phi} = \Phi$ near $\partial B(0, \varrho/2)$ and $\widetilde{\Phi} = \id$ on $B(0, 2\delta)$ and $\widetilde{\Psi} = \Psi$ near $\partial B(0, 2\delta)$ and $\widetilde{\Psi} = \id$ on $B(0, \delta)$. In particular, $\widetilde{\Phi} = H^{-1} \circ F$ on $\partial B(0, \varrho/2)$ and $\widetilde{\Psi} = H$ on $\partial B(0, 2\delta)$.

Eventually, we set
$$
\widetilde{F}(x) = \begin{cases}
    F(x) & \text{ for } x \in \overbar{B}(0, \varrho) \setminus \overbar{B}(0, \varrho/2),\\
    H\circ\widetilde{\Phi}(x) & \text{ for } x \in \overbar{B}(0, \varrho/2) \setminus \overbar{B}(0, 2\delta), \\
    \widetilde{\Psi}(x) & \text{ for } x \in \overbar{B}(0, 2\delta).
\end{cases}
$$
Since $\widetilde{F}(\overbar{B}(0, \varrho/2) \setminus B(0, 2\delta)) = A$ and $H \circ \widetilde{\Phi}$ agrees with $F$ on $\partial B(0, \varrho/2)$, and $\widetilde{\Psi} = H=H\circ\widetilde{\Phi}$ on $\partial B(0, 2\delta)$, $\widetilde{F}$ is indeed a homeomorphism. Clearly, it satisfies the required properties.

If $F$ is additionally assumed to be bi-Lipschitz, then by Lemma~\ref{T20}, $H$ is a bi-Lipschitz homeomorphism and so are the radial extensions $\Phi$ and $\Psi$. Therefore, by Lemma~\ref{T18}, $\widetilde{\Phi}$ and $\widetilde{\Psi}$ are bi-Lipschitz and so is $\widetilde{F}$.
\end{proof}

\section{Proof of Theorems~\ref{T2} and~\ref{T8}}
\label{Proof23}

Now, we prove the remaining two main theorems, about bi-Lipschitz homeomorphisms and homeomorphisms. We begin by applying Palais' trick from Section~\ref{trick} to this setting. This is easy; the difficulty was in proving Lemma~\ref{T19} in the previous section.
\begin{lemma}
\label{T21}
Suppose that $H:\overbar{B}(0, \varrho) \to \bbbr^n$, $H(0) = 0$, is an orientation preserving (bi-Lipschitz) homeomorphism which can be extended as a (bi-Lipschitz) homeomorphism on a neighborhood of $\overbar{B}(0, \varrho)$. Then, for any $\eps > 0$, there is a (bi-Lipschitz) homeomorphism $\widetilde{H}:\bbbr^n \twoheadrightarrow \bbbr^n$ such that
\begin{equation}
\label{eq10}
\widetilde{H}(x)=
\begin{cases} 
H(x) &\text{ if } x\in\overbar{B}(0,\varrho),\\
x &\text{ if } \dist(x,A)\geq\eps,
\end{cases}
\end{equation}
where $A=\overbar{B}(0,\varrho)\cup H(\overbar{B}(0,\varrho))$.
\end{lemma}
\begin{proof}
The proof of this lemma is a~verbatim copy of Lemma~\ref{T3} with `homeomorphism' or `bi-Lipschitz homeomorphism' instead of `diffeomorphism'. Also, instead of Lemma~\ref{T15}, we use Lemma~\ref{T19}.
\end{proof}

\begin{corollary}
\label{T22}
Let $G: \overbar{B}(a, r) \twoheadrightarrow \overbar{B}(a, r)$ be an orientation preserving (bi-Lipschitz) homeomorphism. Then, for any $\eps > 0$, there is a (bi-Lipschitz) homeomorphism $\widetilde{G}: \bbbr^n \twoheadrightarrow \bbbr^n$ such that $\widetilde{G} = G$ on $\overbar{B}(a, r)$ and $\widetilde{G} = \id$ outside $B(a, r + \eps)$.
\end{corollary}
\begin{proof}
The proof is the same as the proof of Corollary~\ref{T7} when `diffeomorphism' is replaced with `(bi-Lipschitz) homeomorphism'. Also, we use Lemma~\ref{T21} in place of Lemma~\ref{T3}. To do so, one needs to check that $G_1$ (defined in the proof of Corollary~\ref{T7}) can be extended as a (bi-Lipschitz) homeomorphism on a neighborhood of $\overbar{B}(0, r)$. This is indeed true, as it suffices to extend $G$ radially to the exterior of $\overbar{B}(a,r)$, cf.\ Remark~\ref{T30}.

\end{proof}

Next, we proceed with the proof of Theorems~\ref{T2} and~\ref{T8} following closely that of Theorem~\ref{T1}.
\begin{lemma}
\label{T23}
Let $\{D_i\}_{i=1}^\ell$, $D_i\subset\bbbr^n$, be a family of pairwise disjoint flat topological (or flat bi-Lipschitz) closed balls and let $p_i\in\mathring{D}_i$, $i=1,\ldots,\ell$, be given. Then there is $\eps_o >0$ such that for any $\eps\in (0,\eps_o)$ there is a (bi-Lipschitz) homeomorphism $F_\eps:\bbbr^n\twoheadrightarrow\bbbr^n$ satisfying $F_\eps(\overbar{B}(p_i,\eps))=D_i$ for $i=1,\ldots,\ell$ and 
$$
F_\eps(x)=x
\text{ if } \dist\Big(x,\bigcup \nolimits_{i=1}^\ell D_i\Big)\geq\eps.
$$
\end{lemma}
\begin{proof}
The proof is the same as the proof of Lemma~\ref{T4} with `(bi-Lipschitz) homeomorphism' in place of `diffeomorphism'. Instead of Lemma~\ref{T3}, we use Lemma~\ref{T21}.
\end{proof}

As in Section~\ref{Proof1}, combining Lemma~\ref{T5} with Lemma~\ref{T23} yields the Euclidean version of Theorems \ref{T2} and \ref{T8}.
\begin{corollary}
\label{T24}
Let $U\subset\bbbr^n$ be a domain.
Suppose that $\{D_i\}_{i=1}^\ell$ and $\{D'_i\}_{i=1}^\ell$, $D_i, D'_i\subset U$, are two families of pairwise disjoint flat topological (or flat bi-Lipschitz) closed balls. If $F_i:D_i\twoheadrightarrow D'_i$, $i=1,2,\ldots,\ell$, are orientation preserving (bi-Lipschitz) homeomorphisms, then there is a (bi-Lipschitz) homeomorphism $F:\bbbr^n\twoheadrightarrow\bbbr^n$ such that $F|_{D_i}=F_i$ and $F=\id$ in $\bbbr^n\setminus U$.
\end{corollary}
\begin{proof}
The proof follows the proof of Corollary~\ref{T6}. One needs to write `(bi-Lipschitz) homeomorphism' instead of `diffeomorphism' and invoke Lemma~\ref{T23} in place of Lemma~\ref{T4} and Corollary~\ref{T22} in place of Corollary~\ref{T7}.
\end{proof}

\begin{proof}[Proof of Corollary~\ref{T28}]
The proof is the same as the proof of Corollary~\ref{T13} except for the fact that it uses Corollary~\ref{T24} in place of Corollary~\ref{T6}. 
\end{proof}

\begin{proof}[Proof of Theorems \ref{T2} and \ref{T8}]
Like in the case of Theorem~\ref{T1}, we prove these results by reducing them to the Euclidean setting, i.\,e., to the case treated in Corollary~\ref{T24}.

Let $U$ be a domain containing $D_i, D'_i$ for $i = 1, \ldots, \ell$. Let $K\subset U$ be a flat topological (or flat bi-Lipschitz) closed ball. Then, there are (bi-Lipschitz) homeomorphisms $H,H':\mathcal{M}^n\twoheadrightarrow\mathcal{M}^n$ such that $H(D_i)\subset\mathring{K}$ and $H'(D_i')\subset\mathring{K}$ for all $i=1,2,\ldots,\ell$, and the homeomorphisms $H$ and $H'$ are identity outside $U$.

Below, we sketch the construction of $H$; homeomorphism $H'$ is constructed in the same way. 
We have to modify the construction of $H$ from Theorem~\ref{T1} since now we clearly cannot directly use a $1$-parameter group of diffeomorphisms. 

For each $i=1,2,\ldots,\ell$, we create a chain of coordinate systems 
$\{(\Phi_{ij},U_{ij})\}_{j=1}^{N_i}$
in $U$ that connects $D_i$ to $\mathring{K}$, i.\,e.,
$$
\Phi_{ij}:U_{ij}\twoheadrightarrow\bbbb^n,
\quad
\overbar{U}_{ij}\subset U,
\quad
D_i\subset U_{i1},
\quad
U_{ij}\cap U_{i,j+1}\neq\varnothing,
\quad
U_{iN_i}\cap\mathring{K}\neq\varnothing,
$$
and $\Phi_{ij}$ is a (bi-Lipschitz) homeomorphism. 

We can also guarantee that the chains corresponding to different balls are disjoint i.e.,
$$
\Big(\bigcup \nolimits_{j=1}^{N_s} U_{sj}\Big)\cap \Big(\bigcup \nolimits_{j=1}^{N_t} U_{tj}\Big)=\varnothing
\qquad
\text{whenever } s\neq t.
$$

Now for each $i$, we construct a (bi-Lipschitz) homeomorphism $H_i:\mathcal{M}^n\twoheadrightarrow\mathcal{M}^n$ that is identity outside $\bigcup_{j=1}^{N_i} U_{ij}$ and such that $H_i(D_i)\subset\mathring{K}$.
We construct $H_i$ as a composition of (bi-Lipschitz) homeomorphisms $H_i=H_{iN_i}\circ\ldots\circ H_{i1}$ such that
$$
H_{i1}(D_i)\subset U_{i2},
\quad
H_{i2}(H_{i1}(D_i))\subset U_{i3},\
\ldots,\
H_{iN_i}(\ldots(H_{i2}(H_{i1}(D_i))))\subset \mathring{K}.
$$
Each such homeomorphism $H_{ij}$ can be constructed in the coordinate system $(\Phi_{ij},U_{ij})$ using a $1$-parameter group of diffeomorphisms.
Note that the homeomorphisms $H_i$ glue to a (bi-Lipschitz) homeomorphism $H:\mathcal{M}^n\twoheadrightarrow\mathcal{M}^n$ with required properties.

The remaining part of the proof is the same as in the proof of Theorem~\ref{T1} except that we use Corollary~\ref{T24} in place of Corollary~\ref{T6}.
\end{proof}

\end{document}